\documentclass[12pt, a4paper]{amsart}

\usepackage{amsmath, amssymb, amsfonts, amsthm}
\usepackage[margin=2cm]{geometry}
\usepackage[]{setspace}
\usepackage[usenames,dvipsnames]{xcolor}
\usepackage[unicode]{hyperref}
\hypersetup{
    colorlinks=true,
    linkcolor= blue,
    citecolor=blue
}
\usepackage{comment, stmaryrd}
\usepackage[all]{xy}
\usepackage[T1]{fontenc}
\usepackage[utf8]{inputenc}

\theoremstyle{plain}
\newtheorem{theorem}{Theorem}[section]
\newtheorem{proposition}[theorem]{Proposition}
\newtheorem{lemma}[theorem]{Lemma}
\newtheorem{corollary}[theorem]{Corollary}

\theoremstyle{definition}
\newtheorem{example}[theorem]{Example}

\newtheorem{question}[theorem]{Question}
\theoremstyle{remark}
\newtheorem{remark}[theorem]{Remark}

\newcommand{\mfp}{\mathfrak{p}}
\newcommand{\mfP}{\mathfrak{P}}
\newcommand{\mfq}{\mathfrak{q}}
\newcommand{\mfQ}{\mathfrak{Q}}
\newcommand{\mcS}{\mathcal{S}}
\newcommand{\Br}{\mathrm{Br}}
\newcommand{\inv}{\mathrm{inv}}

\DeclareMathOperator{\res}{res}

\title[A $p$-adic analogue of Siegel's theorem on sums of squares]{A $p$-adic analogue of Siegel's theorem on sums of squares}
\author{Sylvy Anscombe, Philip Dittmann, and Arno Fehm}
\address{Jeremiah Horrocks Institute, University of Central Lancashire, Preston PR1 2HE, United Kingdom}
\email{sanscombe@uclan.ac.uk}
\address{Afdeling Algebra, KU Leuven, Celestijnenlaan 200b, 3001 Leuven, Belgium}
\email{philip.dittmann@kuleuven.be}
\address{Institut f\"{u}r Algebra, Technische Universit\"{a}t Dresden, 01062 Dresden}
\email{arno.fehm@tu-dresden.de}
\thanks{\today}

\begin{document}
\begin{abstract}
Siegel proved that every totally positive element
of a number field $K$ is the sum of four squares,
so in particular the Pythagoras number is uniformly bounded across number fields.
The $p$-adic Kochen operator provides a $p$-adic analogue of squaring,
and a certain localisation of the ring generated by this operator consists of precisely the totally $p$-integral elements of $K$.
We use this to formulate and prove a $p$-adic analogue of Siegel's theorem, by introducing the $p$-Pythagoras number of a general field, and showing that this number is uniformly bounded across number fields. 
We also generally study fields with finite $p$-Pythagoras number
and show that the growth of the $p$-Pythagoras number in finite extensions is bounded.
\end{abstract}
\maketitle

\section{Introduction}

The study of sums of squares has a long history.
In the context of the integers, Fermat, Euler, Lagrange and many others studied which integers are a sum of a certain number of square integers.
The possibly most famous result in this direction is Lagrange's Four Squares Theorem \cite[Theorem 369]{HW79}
that every non-negative integer is the sum of four squares.
In fact, earlier Euler had proved a version of this theorem for $\mathbb{Q}$: every non-negative rational number is the sum of four square rational numbers.
A comprehensive history of these theorems may be found in \cite[Chapter VIII]{Dic20}.
In the other direction, for both $\mathbb{Z}$ and $\mathbb{Q}$ there exist non-negative numbers that cannot be written as a sum of three squares.
The {\bf Pythagoras number} $\pi(F)$
of a field $F$
is the smallest $n$ such that
\begin{align*}
\big\{x_{1}^{2}+...+x_{m}^{2}\;\big|\;x_{1},...,x_{m}\in F, m\in\mathbb{N}\big\}&=\big\{x_{1}^{2}+...+x_{n}^{2}\;\big|\;x_{1},...,x_{n}\in F\big\}.
\end{align*}
Using this terminology, Euler's theorem becomes the statement that $\pi(\mathbb{Q})=4$.
The following generalization of Euler's theorem 
was conjectured by Hilbert 
and proven by Siegel in \cite{Sie21}, cf.~\cite[Chapter 7, \S1, 1.4]{Pfi95}:

\begin{theorem}[Siegel]
For all number fields $F$, $\pi(F)\leq4$.
\end{theorem}

The study of the Pythagoras number of a field is intimately related to the study of the orderings on that field,
since by a theorem of Artin and Schreier 
the sums of squares 
are precisely the totally positive elements.
In a number field $F$, 
these can be described simply as those elements
that are mapped to $\mathbb{R}_{\geq0}$
by every embedding of $F$ into $\mathbb{R}$,
cf.~\cite[Ch.~3 and 7]{Pfi95}.

We define and study a $p$-adic version of the Pythagoras number, namely the {\bf $p$-Pythagoras number} $\pi_{p}(F)$ of a field $F$, or more generally the $(\mfp,\tau)$-Pythagoras number,
see Section \ref{subsection:p-Pythagoras} for the definition.
Just like the Pythagoras number gives information on the set of totally positive elements, the $p$-Pythagoras number relates to the set of totally $p$-integral elements,
which in a number field $F$ can be described simply as those elements
that are mapped to $\mathbb{Z}_p$ by every embedding of $F$ into $\mathbb{Q}_p$.
Our main result is an inexplicit analogue of Siegel's theorem:

\begin{theorem}\label{thm:main.1}
Let $p$ be a prime number.
There exists $N_{p}\in\mathbb{N}$ such that 
$\pi_{p}(F)\leq N_{p}$
for every number field $F$.
\end{theorem}

This result will be deduced from the more general Theorem \ref{thm:main.2}.
We also give some general results on fields $F$ with finite $(\mfp,\tau)$-Pythagoras number
and prove in Theorem \ref{thm:p_pythagoras_in_fin_extn} that the growth of the $(\mfp,\tau)$-Pythagoras number is bounded in finite extensions.
As an application, we show
in Corollary \ref{cor:holomorphy.diophantine}
that for every open-closed subset of the $p$-adic spectrum of $F$, the associated holomorphy ring is diophantine.
A further application 
can be found in the forthcoming work \cite{ADF3}, 
in which we 
use the results of this paper to
show that 
rings of formal power series over number fields are $\mathbb{Z}$-diophantine in their quotient fields.

\section{The \texorpdfstring{$(\mfp,\tau)$}{(p,τ)}-Pythagoras number}
\label{sec:ppi}

\subsection{\texorpdfstring{$p$}{p}-valuations}
\label{subsection:p-valuations}

A (Krull) valuation $v$ on a field $F$ is a {\bf $p$-valuation} 
if it has a finite residue field $\bar{F}_v$ of characteristic $p$
and value group $v(F^\times)$ such that
the interval $(0,v(p)]$ is finite.
A (finite) {\bf prime} $\mfP$ of a field $F$ 
is an equivalence class of $p$-valuations on $F$
(for the usual notion of equivalence of valuations), 
for some prime number $p$.
We write $v_{\mfP}$ for a representative of $\mfP$
which has $\mathbb{Z}$ as smallest non-trivial convex subgroup of the value group.
See \cite{PR84} for basics regarding $p$-valuations,
and \cite{Feh13} for details on this notion of prime
and some of the following definitions.

\begin{example}
The primes of a number field $K$ correspond precisely to the finite places in the usual sense and we will identify them.
If $K=\mathbb{Q}$ and $p$ is a prime number then $v_{p}$ denotes the usual $p$-adic valuation,
and we denote the corresponding prime also by $p$.
\end{example}

For the rest of this work
we fix a triple
$(K,\mfp,\tau)$,
where $K$ is a number field,
$\mfp$ is a finite prime of $K$,
and $\tau$ is a pair of natural numbers $(e,f)\in\mathbb{N}^{2}$.
We denote by $t_{\mfp}$ a uniformizer of $v_{\mfp}$,
i.e.\ an element with $v_\mfp(t_\mfp)=1$,
we let $q$ denote 
the size of the residue field $\bar{K}_{v_{\mfp}}$.

For a field extension $F/K$ with $\mfP$ a prime of $F$ lying above $\mfp$,
the {\bf relative initial ramification} is
$e(\mfP|\mfp):=v_\mfP(t_\mfp)$, 
the {\bf relative residue degree} is
$f(\mfP|\mfp):=[\bar{F}_{v_{\mfP}}:\bar{K}_{v_{\mfp}}]$,
and the pair $(e(\mfP|\mfp),f(\mfP|\mfp))$ is the {\bf relative type} of
$\mfP$ over $\mfp$.
We say $\mfP$ is of relative type {\bf at most} $\tau$ if
$e(\mfP|\mfp)$ is no greater than $e$,
and
$f(\mfP|\mfp)$ divides $f$.
Likewise, for $\tau'=(e',f')$ we write $\tau \leq \tau'$ if $e \leq e'$ and $f \mid f'$.
We denote by $\mathcal{S}(F)$ the set of primes of $F$,
by $\mathcal{S}^{*}_{\mfp}(F)\subseteq\mathcal{S}(F)$
the set of those primes $\mfP$ of $F$ lying above $\mfp$,
and by $\mathcal{S}_{\mfp}^{\tau}(F)\subseteq\mathcal{S}_\mfp^*(F)$ the subset of those primes $\mfP$ of $F$ 
 which are of relative type at most $\tau$ over $\mfp$.
The corresponding
{\bf holomorphy ring} is
\begin{align*}
R_{\mfp}^{\tau}(F)&:=\bigcap_{\mfP\in S_{\mfp}^{\tau}(F)}\mathcal{O}_{\mfP},
\end{align*}
where
$\mathcal{O}_{\mfP}$
is the valuation ring of $\mfP$,
and
\begin{align*}
\Gamma_{\mfp}^{\tau}(F)&:=\bigg\{\frac{a}{1+t_{\mfp}b}\;\bigg|\;a,b\in\mathcal{O}_{\mfp}[\gamma_{\mfp, t_\mfp}^{\tau}(F)], 1+t_{\mfp}b \neq 0 \bigg\},
\end{align*}
is the corresponding {\bf Kochen ring}, where
\begin{align*}
    \gamma_{\mfp, t_\mfp}^{\tau}(X)&:=\frac{1}{t_{\mfp}}\cdot\bigg(\frac{X^{q^{f}}-X}{(X^{q^{f}}-X)^{2}-1}\bigg)^{e}
\end{align*}
is the {\bf Kochen operator}.
Here and in what follows, if $\gamma\in F(X)$ is a rational function, we mean by $\gamma(F)$
the image of $\gamma$ on $F\setminus\{\mbox{poles of }\gamma\}$.
Note that $\Gamma_{\mfp}^\tau(F)$ does not depend on the choice of $t_\mfp$, since the quotient of two uniformizers of $v_\mfp$ is an element of $\mathcal{O}_\mfp^\times$.
Recall that 
$R_{\mfp}^{\tau}(F)$
is the integral closure of 
$\Gamma_{\mfp}^{\tau}(F)$,
with equality
in the case $e=1$,
see
\cite[Corollary 6.9]{PR84} and the subsequent discussion for more details.

\begin{example}
If $\mfp$ is any place of the number field $K$,
we denote by $K_\mfp$ the completion of $K$ with respect to $\mfp$.
If $\mfp$ is a finite place,
then $K_\mfp$ is a non-archimedean local field and $\mfp$ extends to a unique prime $\mfP$ of $K_\mfp$ of the same type,
so $R_\mfp^\tau(K_\mfp)=R_\mfp^{(1,1)}(K_\mfp)=\mathcal{O}_\mfP$.
In fact, any non-archimedean local field $E$ of characteristic zero carries a unique prime, whose valuation ring we denote by $\mathcal{O}_E$, cf.~\cite[Theorem 6.15]{PR84}.
We say that an extension of non-archimedean local fields is of relative type
at most $\tau$ if this is true for the respective primes.
\end{example}

\subsection{The \texorpdfstring{$(\mfp,\tau)$}{(p,τ)}-Pythagoras number}
\label{subsection:p-Pythagoras}

Let $F/K$ be an extension.
For $g\in\mathcal{O}_{\mfp}[X_{1},...,X_{n}]$,
we write
\begin{align*}
R_{\mfp,g, t_\mfp}^\tau(F)&:=\bigg\{\frac{a}{1+t_{\mfp}b}\;\bigg|\;a,b\in g(\gamma_{\mfp, t_\mfp}^{\tau}(F),...,\gamma_{\mfp, t_\mfp}^{\tau}(F)), 1+t_{\mfp}b\neq0\bigg\},
\end{align*}
and for $n\geq1$
\begin{align*}
R_{\mfp,g,t_\mfp,n}^{\tau}(F)&:=\bigg\{x\in F\;\bigg|\;\text{$x^{m}+a_{m-1}x^{m-1}+...+a_{0}=0$ with $1\leq m\leq n,a_{0},...,a_{m-1}\in R_{\mfp,g,t_\mfp}^{\tau}(F)$}\bigg\}.
\end{align*}
We denote by $\mathcal{P}_{\mfp,n}$ the finite set of those $g\in\mathcal{O}_{\mfp}[X_{1},...,X_{n}]$ of 
degree and height at most $n$ (cf.~\cite[Def.~1.6.1]{BG}).
We write
\begin{align*}
R_{\mfp,n}^{\tau}(F)&:=\bigcup_{t_\mfp} \bigcup_{g\in\mathcal{P}_{\mfp,n}}R_{\mfp,g,t_\mfp,n}^{\tau}(F),
\end{align*}
where $t_\mfp$ varies over those (finitely many) elements of 
the ring of integers $\mathcal{O}_K$ which are uniformizers for $\mfp$ of minimal height.
Then $\big(R_{\mfp,n}^{\tau}(F)\big)_{n\in\mathbb{N}}$ is an increasing chain of subsets of $F$ and
\begin{align*}
R_{\mfp}^{\tau}(F)&=\bigcup_{n\in\mathbb{N}}R_{\mfp,n}^{\tau}(F).
\end{align*}
The {\bf $(\mfp,\tau)$-Pythagoras number} $\pi_{\mfp}^{\tau}(F)$
of $F$
is the smallest $n$ such that
\begin{align*}
R_{\mfp}^{\tau}(F)&=R_{\mfp,n}^{\tau}(F),
\end{align*}
and we write $\pi_{\mfp}^{\tau}(F)=\infty$ if there is no such $n$.
In other words, 
\begin{align*}
\pi_{\mfp}^{\tau}(F)&:=\inf\bigg\{n\in\mathbb{N}\;\bigg|\;R_{\mfp}^{\tau}(F)=R_{\mfp,n}^{\tau}(F)\bigg\}\in\mathbb{N}\cup\{\infty\}.
\end{align*}
In the case $K=\mathbb{Q}$, $\mfp=p$ and $\tau=(1,1)$, 
we write $R_p(F)$ and $\pi_p(F)$,
omitting the relative type $(1,1)$,
and we speak of the $p$-Pythagoras number.
We also write $\gamma_p:=\gamma_{p,p}^{(1,1)}$,
and note that the only two uniformizers (of the prime $p$) in $\mathbb{Z}$ of minimal height are $p$ and $-p$,
with $\gamma_{p,-p}^{(1,1)}=-\gamma_{p}$.
\begin{example}
Since $\mathbb{C}$ is algebraically closed and also is not formally $p$-adic, we have
\begin{align*}
R_{p}(\mathbb{C})=\mathbb{C}=\gamma_{p}(\mathbb{C}),
\end{align*}
in particular $\pi_{p}(\mathbb{C})=1$.
\end{example}

\begin{example}
It follows easily from Hensel’s lemma that
\begin{align*}
R_{p}(\mathbb{Q}_{p})&=\mathbb{Z}_{p}=\gamma_{p}(\mathbb{Q}_{p}),
\end{align*}
in particular $\pi_{p}(\mathbb{Q}_{p})=1$.
\end{example}

\begin{example}\label{Ex:PpC}
In {\cite[Lemma 3.02]{Gro87}} it is shown that every so-called
pseudo $p$-adically closed field $F$ satisfies
\begin{align*}
R_{p}(F)&=\gamma_{p}(F)+\gamma_{p}(F)+\gamma_{p}(F),
\end{align*}
hence $\pi_{p}(F)\leq3$.
This applies for example to the field $\mathbb{Q}^{{\rm t}p}$ of totally $p$-adic algebraic numbers
by a result of Moret-Bailly \cite{MB}.
\end{example}

There are fields $F$ with $\pi(F)=\infty$, see e.g.~\cite[Theorem 1]{Hof99}.
On the other hand, we do not know if $\pi_{p}(F)=\infty$ for any field:

\begin{question}
Is $\pi_p(\mathbb{Q}(X_1,X_2,\dots))=\infty$?
\end{question}

\subsection{Explicit bounds and uniformity in \texorpdfstring{$p$}{p}}
\label{section:pi.Q}

We now prove a few rather elementary statements about $\pi_p(\mathbb{Q})$.
We will drop the relative type $\tau=(1,1)$ from all notation.
Let $\ell$ be a prime number distinct from $p$.

\begin{lemma}\label{lem:p.l.2}
We have $\gamma_{p}(\mathbb{Q})\subseteq\mathbb{Z}_{(\ell)}$ if and only if
neither $X^{p}-X+1$ nor $X^{p}-X-1$
has a zero in $\mathbb{F}_{\ell}$.
\end{lemma}

\begin{proof}
Let $x\in\mathbb{Q}$,
recall that $\gamma_p(x)=\frac{1}{p}((x^p-x)-(x^p-x)^{-1})^{-1}$ and denote by $v_\ell$ the $\ell$-adic valuation.
If $v_\ell(x^p-x)<0$ or $v_\ell(x^p-x)>0$, then $v_\ell(\gamma_p(x))>0$.
If $v_\ell(x^p-x)= 0$, then $x\in\mathbb{Z}_{(\ell)}$, and $v_\ell(\gamma_p(x))<0$ iff $(x^p-x)-(x^p-x)^{-1}\equiv 0\mod{\ell}$,
which means that $x^p-x\equiv\pm1\mod\ell$.
\end{proof}

\begin{proposition}\label{prp:p.l.1}
$\mathbb{Z}[\gamma_{p}(\mathbb{Q})]\subsetneqq\mathbb{Z}_{(p)}$. 
\end{proposition}
\begin{proof}
There exists a prime number $\ell \neq p$ such that $\mathbb{Z}[\gamma_{p}(\mathbb{Q})]$ is contained in $\mathbb{Z}_{(l)}$ by Lemma~\ref{lem:p.l.2}:
specifically, the criterion given there is satisfied by $\ell=2$ if $p$ is odd and by $\ell=17$ for $p=2$.
\end{proof}

\begin{lemma}\label{lem:p.l.1}
If $\ell-1\mid p-1$ then $\gamma_{p}(\mathbb{Q})\subseteq \ell\mathbb{Z}_{(\ell)}$.
\end{lemma}
\begin{proof}
If $\ell-1|p-1$, then $x^{p}-x=0$ for all $x\in\mathbb{F}_{\ell}$. Thus $v_\ell(\gamma_p(x)) > 0$ for all $x \in \mathbb{Q}$, where $v_\ell$ is the $\ell$-adic valuation.
\end{proof}

\begin{proposition}\label{prp:p.l.3}
For every finite set $\mathcal{P}\subseteq\mathbb{Q}[X_{1},X_{2},...]$, 
there exist some $p$ and $\ell\neq p$
with
$$
    \bigcup_{g\in\mathcal{P}}R_{p,g,p}(\mathbb{Q})\subseteq\mathbb{Z}_{(\ell)}.
$$
In particular, $\sup_{p}\pi_{p}(\mathbb{Q})=\infty$.
\end{proposition}
\begin{proof}
Choose $\ell>|\mathcal{P}|+1$ such that $\mathcal{P}\subseteq\mathbb{Z}_{(\ell)}[X_{1},X_{2},...]$.
There exists $a\in\mathbb{Z}$ such that
$a\not\equiv 0\pmod\ell$ and
$a\not\equiv g(0,...,0)\pmod\ell$
for every $g\in\mathcal{P}$.
By Dirichlet's theorem on primes in arithmetic progressions
(see \cite[VII, (13.2)]{Neu99}),
there exist infinitely many primes $p>\ell$ with
$p\equiv1\pmod{\ell-1}$ and $p\equiv -a^{-1}\pmod{\ell}$.
Then
$$
    g(\gamma_{p}(\mathbb{Q}),...,\gamma_{p}(\mathbb{Q}))\subseteq g(0,...,0)+\ell\mathbb{Z}_{(\ell)}
$$
by Lemma \ref{lem:p.l.1},
hence $1+pg(\gamma_{p}(\mathbb{Q}),...,\gamma_{p}(\mathbb{Q}))\subseteq\mathbb{Z}_{(\ell)}^{\times}$
by the choice of $a$ and $p$.
Thus $R_{p,g,p}(\mathbb{Q})\subseteq\mathbb{Z}_{(\ell)}$ for every $g\in\mathcal{P}$.

By the integral closedness of $\mathbb{Z}_{(\ell)}$ this implies
$R_{p,g,p,n}(\mathbb{Q})\subseteq\mathbb{Z}_{(\ell)}$
for every $n$.
Note that $R_{p,g,-p,n}(F) =  -R_{p,g^*,p,n}(F)$, where 
$g^*(X_1,\dots,X_n) = -g(-X_1,\dots,-X_n)$ has the same height as $g$.
Therefore, 
applying the above to the set $\mathcal{P}$
of all $f\in\mathbb{Q}[X_1,\dots,X_n]$ of degree and height at most $n$,
we obtain $\ell$ and $p>\ell$ with
$$
 \bigcup_{g\in\mathcal{P}_{p,n}}(R_{p,g,p,n}(F)\cup R_{p,g,-p,n}(F))
 \subseteq\bigcup_{p\in\mathcal{P}}R_{p,g,p,n}(F)\subseteq\mathbb{Z}_{(\ell)},
$$
and therefore $\pi_p(\mathbb{Q})>n$.
\end{proof}

\subsection{The Kochen operator}
\label{subsection:Kochen}

For later use, we explore several simple properties of the Kochen operator.
Let $F/K$ be any extension.

\begin{lemma}\label{lem:gamma.1}
Let $\mfP\in\mcS_{\mfp}^{*}(F)$ and 
suppose that $x\in F$ is not a pole of $\gamma_{\mfp,t_\mfp}^\tau$.
Then
\begin{align*}
    v_{\mfP}(\gamma_{\mfp, t_\mfp}^{\tau}(x))&=\left\{\begin{array}{ll}
        -eq^{f}v_{\mfP}(x)-v_{\mfP}(t_{\mfp})&\text{if $v_{\mfP}(x)<0$},\\
        ev_{\mfP}(x)-v_{\mfP}(t_{\mfp})&\text{if $v_{\mfP}(x)>0$},\\
        ev_{\mfP}(x^{q^{f}}-x)-v_{\mfP}(t_{\mfp})&\text{if $v_{\mfP}(x)=0$ and $v_{\mfP}(x^{q^{f}}-x)>0$},\\
        -ev_{\mfP}((x^{q^{f}}-x)^{2}-1)-v_{\mfP}(t_{\mfp})&\text{if $v_{\mfP}(x)=0$ and
        $v_{\mfP}(x^{q^{f}}-x)=0$}.
    \end{array}\right.
\end{align*}
\end{lemma}
\begin{proof}
This is a matter of calculating valuations.
\end{proof}

\begin{lemma}\label{lem:gamma}
Let $\mfP\in\mathcal{S}_{\mfp}^{*}(F)$.
Suppose that $x\in F$ is not a pole of $\gamma_{\mfp, t_\mfp}^\tau$ and satisfies
either
\begin{enumerate}
\item[(i)]
$0<(e+1)v_{\mfP}(x)\leq v_{\mfP}(t_{\mfp})$, or
\item[(ii)]
$v_{\mfP}(x)=0$ and
$[\mathbb{F}_{q}(\res_\mfP(x)):\mathbb{F}_{q}]\nmid f$, where $\res_\mfP(x)$ is the residue of $x$.
\end{enumerate}
Then 
\begin{align*}
v_\mathfrak{P}(\gamma_{\mfp, t_\mfp}^{\tau}(x))&\leq-\frac{1}{e+1}v_\mathfrak{P}(t_{\mathfrak{p}})<0.
\end{align*}
\end{lemma}

\begin{proof} 
In case (i),
Lemma \ref{lem:gamma.1} gives that
\begin{align*}
v_\mfP(\gamma_{\mfp, t_\mfp}^{\tau}(x))&=ev_\mfP(x)-v_\mfP(t_{\mathfrak{p}})
\leq-\frac{1}{e+1}v_\mfP(t_{\mathfrak{p}}).
\end{align*}
In case (ii), the residue of $x$ is not a root of $X^{q^{f}}-X$,
and so
$$
v_\mfP(\gamma_{\mfp,t_\mfp}^{\tau}(x))=-ev_\mfP((x^{q^{f}}-x)^{2}-1)-v_\mfP(t_{\mathfrak{p}})
\leq-v_\mfP(t_{\mathfrak{p}})
\leq-\frac{1}{e+1}v_\mfP(t_{\mathfrak{p}}),
$$
also by Lemma \ref{lem:gamma.1}.
\end{proof}

\begin{lemma}\label{lem:gamma.3}
Let $\mfP\in\mcS_{\mfp}^{*}(F)$ and $x,y\in F$,
and suppose that $x$ is not a pole of $\gamma_{\mfp,t_\mfp}^\tau$,
and
$v_{\mfP}(\gamma_{\mfp, t_\mfp}^\tau(x))<0$.
If
$v_{\mfP}(x-y)\geq v_{\mfP}(t_\mfp)$,
then also $y$ is not a pole of $\gamma_{\mfp,t_\mfp}^\tau$, and $v_{\mfP}(\gamma_{\mfp, t_\mfp}^\tau(y))<0$.
\end{lemma}
\begin{proof}
If $v_\mfP(x)\leq 0$, then 
in particular $v_\mfP(x)< v_\mfP(t_\mfp)$,
while if $v_\mfp(x)>0$,
then $v_{\mfP}(\gamma_{\mfp, t_\mfp}^{\tau}(x))=ev_{\mfP}(x)-v_{\mfP}(t_{\mfp})$ by Lemma \ref{lem:gamma.1},
hence $v_{\mfP}(\gamma_{\mfp, t_\mfp}^{\tau}(x))<0$
implies that $v_{\mfP}(x)<v_{\mfP}(t_{\mfp})$
also in this case.
Therefore, in either case
we conclude from 
$v_{\mfP}(x-y)\geq v_{\mfP}(t_\mfp)$
that
$v_{\mfP}(x)=v_{\mfP}(y)$.
We make a case distinction:

Suppose first that $v_{\mfP}(x)\neq0$.
By Lemma \ref{lem:gamma.1},
in this case,
$v_{\mfP}(\gamma_{\mfp, t_\mfp}^{\tau}(x))$ depends only on $v_{\mfP}(x)$.
Therefore $v_{\mfP}(\gamma_{\mfp, t_\mfp}^{\tau}(y))=v_{\mfP}(\gamma_{\mfp, t_\mfp}^{\tau}(x))<0$.

Suppose now that $v_\mfP(x)=0$.
As $x-y$ divides $x^{q^f}-y^{q^f}$ in $\mathcal{O}_\mfP$,
we have that
$v_{\mfP}(y^{q^f}-y-x^{q^f}+x)\geq v_{\mfP}(x-y)\geq v_{\mfP}(t_{\mfp})$.
If 
$v_\mfP(x^{q^f}-x)=0$, then in particular $v_\mfP(x^{q^f}-x)<v_\mfP(t_\mfp)$,
while if $v_\mfP(x^{q^f}-x)>0$, then
$v_{\mfP}(\gamma_{\mfp, t_\mfp}^{\tau}(x))<0$
implies that $v_{\mfP}(x^{q^{f}}-x)<\frac{1}{e}v_{\mfP}(t_{\mfp})\leq v_{\mfP}(t_{\mfp})$
by Lemma \ref{lem:gamma.1}.
Thus $v_\mfP(y^{q^f}-y)=v_\mfP(x^{q^f}-x)$ in both cases.
If $v_{\mfP}(x^{q^{f}}-x)=0$, then
Lemma~\ref{lem:gamma.1} gives immediately that
$v_{\mfP}(\gamma_{\mfp, t_\mfp}^{\tau}(y))<0$,
while if $v_{\mfP}(x^{q^{f}}-x)>0$,
then
Lemma~\ref{lem:gamma.1} shows that
$v_{\mfP}(\gamma_{\mfp, t_\mfp}^{\tau}(x))$ depends only on $v_{\mfP}(x^{q^{f}}-x)$,
hence $v_{\mfP}(\gamma_{\mfp, t_\mfp}^{\tau}(y))=v_{\mfP}(\gamma_{\mfp, t_\mfp}^{\tau}(x))<0$.
\end{proof}

\section{Diophantine families}
\label{section:diophantine}

A {\bf diophantine}
subset of a field $F$
is the 
image of the $F$-rational points of some $F$-variety $V$
under a morphism $V\rightarrow\mathbb{A}^1_F$.
As we want to discuss questions of uniformity we 
use the following slightly more sophisticated notion:
An {\bf $n$-dimensional diophantine family over $K$}
is a map $D$ from the class of field extensions $F$ of $K$ to sets
which is given by
finitely many polynomials $f_1,\dots,f_r\in K[X_1,\dots,X_n,Y_1,\dots,Y_m]$, for some $m$,
in the sense that 
$$ 
D(F)= \{x\in F^n \;|\; \exists y\in F^m:f_1(x,y)=0,\dots,f_r(x,y)=0\}
$$
for every extension $F/K$.
In this case, we say that the polynomials $f_{1},...,f_{r}$ {\bf define} $D$.
Note that if $E/F$ is an extension, then $D(F)\subseteq D(E)$.

\begin{remark}
From the point of view of algebraic geometry,
an $n$-dimensional diophantine family $D$ over $K$ is given by a morphism
of (not necessarily irreducible) $K$-varieties $\varphi:V\rightarrow\mathbb{A}_K^n$
in the sense that $D(F)=\varphi(V(F))$ for every extension $F/K$.
\end{remark}

\begin{remark}\label{rem:modeltheory}
From the point of view of model theory,
an $n$-dimensional diophantine family $D$ over $K$ is given by an 
existential formula $\varphi(x_1,\dots,x_n)$ in the language of rings with free variables among $x_1,\dots,x_n$ and parameters from $K$,
in the sense that for every extension $F/K$, 
$D(F)$ is the set defined by $\varphi$ in $F$,
i.e.~the set of $a\in F^m$ such that
$F\models\varphi(a)$.
Such a formula is equivalent (modulo the theory of fields)
to a formula of the form
$$
 \exists y_1\dots y_m:\bigwedge_{i=1}^r f_i(x_1,\dots,x_n,y_1,\dots,y_m)=0
$$
with $f_1,\dots,f_r\in K[X_1,\dots,X_n,Y_1,\dots,Y_m]$.
\end{remark}

Most of the usual constructions for diophantine sets (see e.g.~\cite{Shl06}) go through for diophantine families:

\begin{lemma}\label{lem:dio.1}
If $D_1,D_2$ are $n$-dimensional diophantine families over $K$,
then there are $n$-dimensional diophantine families $D_1\cup D_2$ 
and $D_1\cap D_2$ over $K$ such that
$(D_1\cup D_2)(F)=D_1(F)\cup D_2(F)$ 
and $(D_1\cap D_2)(F)=D_1(F)\cap D_2(F)$ for every $F/K$.
\end{lemma}
\begin{proof}
Suppose that the polynomials $f_{1},...,f_{r}\in K[X_{1},...,X_{n},Y_{1},...,Y_{m}]$ define $D_{1}$ and
that the polynomials $g_{1},...,g_{s}\in K[X_{1},...,X_{n},Z_{1},...,Z_{l}]$ define $D_{2}$.
We may assume that the variables $Y_{i}$ and $Z_{j}$ are distinct.
We observe that
$f_{1},...,f_{r},g_{1},...,g_{s}$
define $D_{1}\cap D_{2}$.
Slightly less trivially, we have that
$f_{1}g_{1},...,f_{i}g_{j},...,f_{r}g_{s}$
define
$D_{1}\cup D_{2}$.
\end{proof}

\begin{lemma}\label{lem:dio.2}
Suppose that $D_{1}$
and $D_{2}$
are $n_{1}$- respectively $n_2$-dimensional
diophantine families over $K$.
Then there is an $(n_1+n_2)$-dimensional diophantine family $D_1\times D_2$ over $K$ such that $(D_1\times D_2)(F)=D_1(F)\times D_2(F)$ for every $F/K$.
\end{lemma}
\begin{proof}
Suppose that the polynomials
$f_{1},...,f_{r}\in K[X_{1},...,X_{n_1},Y_{1},...,Y_{m}]$ define $D_{1}$ and
that the polynomials
$g_{1},...,g_{s}\in K[X'_{1},...,X'_{n_2},Z_{1},...,Z_{l}]$ define $D_{2}$.
This time, we suppose that all the variables $X_{i},X'_{i},Y_{i},Z_{i}$ are distinct.
Then the polynomials $f_{1},...,f_{r},g_{1},...,g_{s}$ define $D_{1}\times D_{2}$.
\end{proof}

\begin{lemma}\label{lem:dio.rational}
  Let $D$ be an $n$-dimensional diophantine family over $K$
and $f=(\frac{g_1}{h_1},\dots,\frac{g_k}{h_k})$ a tuple of rational functions with $g_i, h_i \in K[X_{1},\dots,X_{n}]$ such that for every $i$ the polynomials $g_i$ and $h_i$ are coprime.
Then there is an $k$-dimensional diophantine family $fD$ with
\[
 (fD)(F) = \left\{\left.\left(\frac{g_1(x)}{h_1(x)}, \dotsc, \frac{g_k(x)}{h_k(x)}\right)\;\right|\; x \in D(F), h_i(x) \neq 0 \text{ for all $i$}\right\}
\] 
for every $F/K$.
\end{lemma}
\begin{proof}
  Let $f_1, \dotsc, f_r \in K[X_1, \dotsc, X_n, Y_1, \dotsc, Y_m]$ define $D$. Then a tuple $(z_1, \dotsc, z_k) \in F^k$ is 
  an element of the right hand side
  if and only if there exists $(x_1, \dotsc, x_n, y_1, \dotsc, y_m, w_1, \dotsc, w_k) \in F^{n+m+k}$ such that
  \begin{enumerate}
  	\item $g_i(x_1, \dotsc, x_n) - z_i h_i(x_1, \dotsc, x_n) = 0$ for all $i =1, \dotsc, k$,
  	\item $w_i h_i(x_1, \dotsc, x_n) = 1$ for all $i=1,\dots,k$, and
  	\item $f_j(x_1, \dotsc, x_n, y_1, \dotsc, y_m) = 0$ for all $j = 1, \dotsc, r$.
  \end{enumerate}
  Each of these conditions is the vanishing of a polynomial 
  in the variables $W_1,\dots,W_k$, $X_1,\dots,X_k$, $Y_1,\dots,Y_r$  and $Z_1,\dots,Z_k$ 
  over $K$.
\end{proof}

\begin{remark}\label{rmk:gamma.diophantine}
Perhaps the most trivial $1$-dimensional diophantine family over $K$ is the one assigning the set $F$ to every field $F/K$.
As described above in Section \ref{subsection:p-valuations}, given a rational function $\gamma\in K(X)$ and a field $F/K$,
we write $\gamma(F)$ to mean the image under $\gamma$ of $F\setminus\{\mbox{poles of }\gamma\}$.
By this small abuse of notation, $\gamma$ may be identified with the map which sends a field $F/K$ to its image $\gamma(F)$ under $\gamma$.
Then by Lemma \ref{lem:dio.rational},
$\gamma$ is a $1$-dimensional diophantine family over $K$.
This applies in particular to the Kochen operator
$\gamma_{\mfp, t_\mfp}^{\tau}$.
\end{remark}

\begin{lemma}\label{lem:dioph_section}
If $D$ is an $n$-dimensional diophantine family over $K$ and $a=(a_1,\dots,a_r)\in K^r$, $r<n$, then there is a $(n-r)$-dimensional family $D_a$ over $K$ with
$$
 D_a(F) = \{x\in F^{n-r} \;|\; (x,a)\in D(F)\}
$$
for every $F/K$.
\end{lemma}
\begin{proof}
Again, let $f_{1},...,f_{r}\in K[X_{1},...,X_{n},Y_{1},...Y_{m}]$
define $D$.
We write
\begin{align*}
g_{i}(X_{1},...,X_{n-r},Y_{1},...,Y_{m})&:=f_{i}(X_{1},...,X_{n-r},a_{1},...,a_{r},Y_{1},...,Y_{m}).
\end{align*}
Then the polynomials $g_{1},...,g_{r}\in K[X_{1},...,X_{n-r},Y_{1},...,Y_{m}]$
define the $(n-r)$-dimensional diophantine family $D_{a}$ over $K$.
\end{proof}

\begin{example}\label{ex:Gamma.n.diophantine}
Each of the $R_{\mfp,n}^{\tau}$ is a $1$-dimensional diophantine family over $K$.
\end{example}

\begin{proposition}\label{prop:compactness}
Let $D,D_1,D_2,\dots$ be $n$-dimensional diophantine families over $K$.
If $D(F)\subseteq\bigcup_{i\in\mathbb{N}}D_i(F)$
for every extension $F/K$,
then there exists $N$ such that 
$D(F)\subseteq\bigcup_{i=1}^ND_i(F)$
for every extension $F/K$.
\end{proposition}

\begin{proof}
In light of Remark \ref{rem:modeltheory},
this is a direct consequence of the compactness theorem of model theory,
see for example \cite[Theorem 2.1.4]{Mar02}.
\end{proof}

\begin{proposition}\label{prop:compactness_bounded_pi}
Let $D$ be a $1$-dimensional diophantine family over $K$
and let $\mathcal{K}$ be a class of extensions of $K$. 
If
\begin{enumerate}
\item[(i)]
$D(L)=R_\mfp^\tau(L)$ for every $L\in\mathcal{K}$, and
\item[(ii)]
$D(E)\subseteq \mathcal{O}_E$
for every finite extension $E/K_\mfp$ of relative type at most $\tau$,
\end{enumerate}
then there exists $N$ such that $\pi_\mfp^\tau(L)\leq N$ for every $L\in\mathcal{K}$.
\end{proposition}

\begin{proof}
Let $F$ be any extension of $K$.
For $\mfP\in\mathcal{S}_\mfp^\tau(F)$
let $(F',\mfP')$ denote a $p$-adic closure of $(F,\mfP)$
(see \cite[\S3]{PR84}).
By the $p$-adic Lefschetz principle, the assumption (ii) implies that
$D(F')\subseteq\mathcal{O}_{\mfP'}$, in particular $D(F)\subseteq\mathcal{O}_{\mfP'}\cap F=\mathcal{O}_\mfP$.
(In model-theoretic terms, $F'$ is elementarily equivalent, in the language of valued fields, to a finite extension $E$
of $K_\mfp$ of relative type at most $\tau$.
More precisely, if $F_0$ denotes the algebraic part of $F'$, then both
$F_0 K_\mfp$ and $F'$ are elementary extensions of $F_0$
by \cite[Theorem 5.1]{PR84}.)
In particular,
$
 D(F)
 \subseteq \bigcap_{\mfP\in\mathcal{S}_\mfp^\tau(F)}\mathcal{O}_\mfP=R_\mfp^\tau(F)
$.
So since $R_\mfp^\tau(F)=\bigcup_{n=1}^{\infty}R_{\mfp,n}^{\tau}(F)$,
by Proposition \ref{prop:compactness} there exists $N$ such that
$D(F)\subseteq\bigcup_{n=1}^{N}R_{\mfp,n}^{\tau}(F)$ for every $F/K$.
In fact $(R_{\mfp,n}^{\tau}(F))_{n\in\mathbb{N}}$ is an increasing chain, so $D(F)\subseteq R_{\mfp,N}^{\tau}(F)$.
Thus for $L\in\mathcal{K}$,
(i) implies that
$R_\mfp^\tau(L)=D(L)\subseteq R_{\mfp,N}^{\tau}(L)$,
which shows that $\pi_\mfp^\tau(L)\leq N$.
\end{proof}

\begin{remark}
We also have the following converse: If $\pi_\mfp^\tau(L) \leq N$ for all $L \in \mathcal{K}$, then $D = R_{\mfp, N}^\tau$ is a diophantine family satisfying both conditions.
This indicates that while our definition of $\pi_\mfp^\tau$ depends on the construction of the height function on polynomials over $\mathcal{O}_\mfp$, the property of a class $\mathcal{K}$ to have bounded $(\mfp, \tau)$-Pythagoras number is a very robust notion and does not depend on the details of the height function.
\end{remark}

\begin{remark}
The notion that a class $\mathcal{K}$ has bounded
$(\mfp,\tau)$-Pythagoras number is robust in a further sense: under taking a suitable alternative for the Kochen operator.
Consider a rational function $\delta\in K(X)$
and suppose that 
$R_{\mfp}^{\tau}(F)$   
is the integral closure in $F$ of the ring
\begin{align*}
    R'(F)&:=\bigg\{\frac{a}{1+t_{\mfp}b}\;\bigg|\;a,b\in\mathcal{O}_{\mfp}[\delta(F)],1+t_{\mfp}b\neq0\bigg\},
\end{align*}
for every extension $F/K$.
We introduce a new $1$-dimensional diophantine family $R_{n}'$ over $K$,
by defining $R_{n}'(F)$ in terms of $\delta$ exactly as $R_{\mfp,n}(F)$ is defined in terms of $\gamma_{\mfp,t_{\mfp}}^{\tau}$.
Then
$$
    R_{\mfp}^{\tau}(F)=\bigcup_{n=1}^\infty R'_{n}(F),
$$
for all $F/K$.
Simply adapting the proof of Proposition \ref{prop:compactness_bounded_pi},
a class $\mathcal{K}$ of extensions of $K$
has bounded $(\mfp,\tau)$-Pythagoras number
if and only if
there is $M\in\mathbb{N}$ such that
$R'_{M}(L)=R_{\mfp}^{\tau}(L)$, for all $L\in\mathcal{K}$.
Also note that at least in the case $\tau=(1,1)$,
the Kochen operator $\gamma_{\mfp,t_\mfp}^\tau$ is universal in the sense that
every such $\delta$ is in fact a rational function in $\gamma_{\mfp,t_\mfp}^\tau$, 
see \cite[Corollary 7.12]{PR84}.
\end{remark}

\section{The \texorpdfstring{$(\mfp,\tau)$}{(p,τ)}-Pythagoras number of number fields}
\label{section:number.fields}

Introduced by Poonen (\cite{Poo09}),
and subsequently used and developed by others including Koenigsmann (\cite{Koe16}) and the second author (\cite{Dit18}),
the following diophantine predicates behave well in local fields, and satisfy a strong local-global principle.
They are defined from central simple algebras.
For further details about central simple algebras, the Brauer group, and associated local-global principles, see \cite[Section 6.3]{NSW}.

Let $A$ be a central simple algebra of prime degree $\ell$ 
over a field $F$.
Following
\cite[Section 2]{Dit18}, we let
\begin{align*}
S_{A}(F)&:=\Big\{\mathrm{Trd}(x)\;\Big|\;x\in A,\mathrm{Nrd}(x)=1\Big\}\subseteq F,
\end{align*}
where $\mathrm{Trd}$ and $\mathrm{Nrd}$ are the reduced norm and reduced trace, see \cite[Construction 2.6.1]{GS06} for details.
We also define
\begin{align*}
T_{A}(F)&:=\left\{\begin{array}{ll}
S_{A}(F)&\text{if $\ell>2$},\\
S_{A}(F)-S_{A}(F)&\text{if $\ell=2$}.
\end{array}\right.
\end{align*}

If $A$ is a central simple algebra over $F$ and $E/F$ is any extension, we view $A_E:=A\otimes_{F}E$ as a central simple algebra over $E$
and write
$S_{A}(E):=S_{A_E}(E)$ and $T_{A}(E):=T_{A_E}(E)$.

\begin{lemma}\label{lem:S.T.dio}
Both $S_{A}$ and $T_{A}$ are $1$-dimensional diophatine families over $F$.
\end{lemma}
\begin{proof}
This is shown in \cite[Lemma 2.12]{Dit18} and the subsequent discussion.
\end{proof}

Recall that $A$ is {\bf split} if it is isomorphic to a matrix algebra over $F$, and $A$ splits over $E$ if $A_E$ is split.
The behaviour of $S_{A}$ and $T_{A}$ in a completion $F$ of a number field $L$ is determined by whether or not $A$ splits over $F$, and the behaviour of $S_{A}$ and $T_{A}$ in $L$ is controlled by a local-global principle,
which leads to the following:

\begin{proposition}[{\cite[Proposition 2.9]{Dit18}}]\label{prp:Philip.Tcomputation}
  Let $L$ be a number field and $A$ a central simple algebra over $L$ of prime degree $\ell$ which splits over all real completions of $L$.
  Then
  \[ T_A(L) = \bigcap_{\mfp} \mathcal{O}_\mfp ,\]
  where the intersection is over the finitely many finite primes $\mfp$ of $L$ such that $A$ does not split over $L_\mfp$.
\end{proposition}

\begin{proposition}[see {\cite[Proposition 2.6]{Dit18}}]\label{prp:Philip.3}
Let $F$ be a non-archimedean local field
of characteristic zero
and let $A$ be a central simple algebra over $F$ of prime degree $\ell$.
If $A$ is non-split then
$T_{A}(F)=\mathcal{O}_{F}$.
\end{proposition}

Note that \cite[Proposition 2.6]{Dit18} is stated for central division algebras of prime degree,
but a non-split central simple algebra of prime degree is a division algebra.

Recall that above we fixed a number field $K$, a finite place $\mfp$ of $K$, and a pair $\tau=(e,f)\in\mathbb{N}^{2}$.
Given this data $(K,\mfp,\tau)$, we now describe a choice of algebras $A,B$ over $K$.

\begin{proposition}\label{prp:AB.1}
For every prime number $\ell$
there exist central simple algebras $A,B$ of degree $\ell$ over $K$ such that
\begin{enumerate}
\item neither of them splits over $K_\mfp$,
\item for every finite place $\mfq\neq\mfp$ of $K$, at least one of them splits over $K_\mfq$,
\item for every infinite place $\mfq$ of $K$, both of them split over $K_\mfq$.
\end{enumerate}
\end{proposition}

\begin{proof}
The Brauer equivalence classes $[A]$ of central simple algebras $A$ over a field $F$ form the Brauer group $\Br(F)$ of $F$,
see \cite[(6.3.2) Definition]{NSW}.
For
an extension $F/K$, there is a group homomorphism
$\Br(K)\longrightarrow\Br(F)$
given by $[A]\longmapsto[A_F]$.
Moreover, the local Hasse invariant is an isomorphism
\begin{align*}\label{eq:inv}\tag{a}
    \inv_{K_{\mfq}}:\Br(K_{\mfq})&\longrightarrow
    \left\{\begin{array}{ll}
    \mathbb{Q}/\mathbb{Z} & \text{if $\mfq$ is finite},\\
    \frac{1}{2}\mathbb{Z}/\mathbb{Z} & \text{if $\mfq$ is infinite and $K_\mfq\cong\mathbb{R}$},\\
    0&\text{if $\mfq$ is infinite and $K_\mfq\cong\mathbb{C}$}
\end{array}\right.
\end{align*}
and so $A$ splits over $K_\mfq$ if and only if ${\rm inv}_{K_\mfq}([A])=0$.
There will be no ambiguity if we write
$\inv_{K_{\mfq}}([A])=\inv_{K_{\mfq}}([A_{K_{\mfq}}])$.
Note that each of the local Hasse invariants
$\inv_{K_{\mfq}}$ takes its values in $\mathbb{Q}/\mathbb{Z}$.

The Albert--Brauer--Hasse--Noether Theorem ({\cite[(8.1.17) Theorem]{NSW}}) gives the exact sequence
\begin{align}\label{eq:ABHN}\tag{b}
0\longrightarrow\Br(K)\xrightarrow{\hspace*{8mm}}\bigoplus_{\mfq\in \mathbb{S}(K)}\Br(K_{\mfq})\xrightarrow{\hspace*{2mm}\inv_{K}\hspace*{2mm}}\mathbb{Q}/\mathbb{Z}\longrightarrow0,
\end{align}
where 
$\mathbb{S}(K)$ is the set of (finite and infinite) places of $K$,
and $\inv_{K}$ is the sum of the local invariant maps
$\inv_{K_{\mfq}}$.

Fix 
two distinct finite places $\mfq_{1},\mfq_{2}\neq\mfp$ of $K$.
We define two sequences
$(a_{\mfq})_{\mfq\in\mathbb{S}(K)}$ and
$(b_{\mfq})_{\mfq\in\mathbb{S}(K)}$
of rational numbers, indexed by the places of $K$, by
\begin{itemize}
\item
$a_{\mfp}=b_{\mfp}=\ell^{-1}$,
\item
$a_{\mfq_{1}}=(\ell-1)\ell^{-1}$ and $b_{\mfq_{1}}=0$,
\item
$a_{\mfq_{2}}=0$ and $b_{\mfq_{2}}=(\ell-1)\ell^{-1}$,
\item
$a_{\mfq}=b_{\mfq}=0$,
for every other place $\mfq$.
\end{itemize}
Note that only finitely many of the elements of these sequences are nonzero.
Thus, by applying the inverses of the local Hasse invariants
from (\ref{eq:inv}),
the sequences $(a_{\mfq})_\mfq$ and $(b_{\mfq})_\mfq$ correspond to elements of the direct sum
$\bigoplus_{\mfq}\Br(K_{\mfq})$.
We also note the sums
$$
\sum_{\mfq\in\mathbb{S}(K)}a_{\mfq}=\sum_{\mfq\in\mathbb{S}(K)}b_{\mfq}=0\quad\mbox{in }\mathbb{Q}/\mathbb{Z}.
$$
By the exactness of the short exact sequence (\ref{eq:ABHN}),
we get (unique)
equivalence classes $[A]$ and $[B]$ in $\Br(K)$
such that
$\inv_{K_{\mfq}}([A])=a_{\mfq}+\mathbb{Z}$ and 
$\inv_{K_{\mfq}}([B])=b_{\mfq}+\mathbb{Z}$,
for all $\mfq\in\mathbb{S}(K)$.
Thus both $[A]$ and $[B]$ are of period $\ell$.
As $K$ is a number field, this implies that they are also of index $\ell$ (\cite[32.19]{Reiner}),
which means that if $A$ and $B$ denote the unique division algebras in $[A]$ respectively $[B]$, 
then these are of degree $\ell$.
\end{proof}

\begin{proposition}\label{prp:AB.2}
Let $\ell$ be a prime number with $\ell>ef$.
If $A$ and $B$ are algebras as in Proposition \ref{prp:AB.1}, then 
\begin{enumerate}
\item[(i)]
for all finite extensions $E/K_{\mfp}$ of relative type at most $\tau$,
\begin{align*}
T_{A}(E)+T_{B}(E)&= \mathcal{O}_{E};
\end{align*}
\item[(ii)]
and for all number fields $L/K$,
\begin{align*}
T_{A}(L)+T_{B}(L)&\supseteq\bigcap_{\mfP\in\mathcal{S}_{\mfp}^{*}(L)}\mathcal{O}_{\mfP}.
\end{align*}
\end{enumerate}
\end{proposition}

\begin{proof}
First, suppose that $E/K_{\mfp}$ is a finite extension of relative type at most $\tau$.
Thus $[E:K_{\mfp}]\leq ef<\ell$,
so since  $A$ and $B$ do not split over $K_\mfp$, 
they also do not split over $E$ by \cite[Corollary 4.5.9]{GS06}.
Therefore we may apply Proposition \ref{prp:Philip.3} to obtain
\begin{align*}
T_{A}(E)+T_{B}(E)&=\mathcal{O}_{E}+\mathcal{O}_{E}=\mathcal{O}_{E}.
\end{align*}

Next, let $L/K$ be any number field and let $\mfQ$ be a prime of $L$ which lies over a prime $\mfq$ of $K$.
If $\mfq \neq \mfp$, then at least one of $A$ and $B$ splits over $K_\mfq$ and therefore also over the completion $L_\mfQ$ by construction.
Hence
\begin{align*}
  T_A(L) + T_B(L) &= \bigcap_{\stackrel{\mfQ\in\mathcal{S}(L)}{A_{L_\mfQ} \text{ not split}}} \mathcal{O}_\mfQ + \bigcap_{\stackrel{\mfQ\in\mathcal{S}(L)}{B_{L_\mfQ} \text{ not split}}} \mathcal{O}_\mfQ 
  = \bigcap_{\stackrel{\mfQ\in\mathcal{S}(L)}{A_{L_\mfQ} \text{ and }B_{L_\mfQ} \text{ not split}}} \mathcal{O}_\mfQ 
  \supseteq \quad \bigcap_{\mfP\in\mathcal{S}_{\mfp}^{*}(L)} \mathcal{O}_\mfP,
\end{align*}
where the first equality is Proposition \ref{prp:Philip.Tcomputation} and the second equality follows from weak approximation (see e.g.~\cite[1.1.3]{EP05}).
\end{proof}

As before, fix a uniformizer $t_\mfp \in K$ of $\mfp$.
For central simple algebras $A,B$ over $K$ and
an extension $F/K$ we define $D_{\mfp,t_\mfp,A,B}^{\tau}(F)$ as
$$
    \bigg\{\frac{x}{1+t_{\mfp}w^{e+1}y}\;\bigg|\;x,y\in T_{A}(F)+T_{B}(F),w\in\gamma_{\mfp, t_\mfp}^\tau(F), 1+t_{\mfp}w^{e+1}y \neq 0 \bigg\}.
$$

\begin{lemma}\label{lem:D.dio}
$D_{\mfp,t_\mfp,A,B}^{\tau}$ is a $1$-dimensional diophantine family over $K$.
\end{lemma}
\begin{proof}
We have seen in Lemma \ref{lem:S.T.dio} that $T_{A}$ and $T_{B}$ are $1$-dimensional diophantine families over $K$.
The claim follows by applying 
Lemma \ref{lem:dio.rational} to
the 5-dimensional diophantine family $T_A\times T_B\times T_A\times T_B\times\gamma_{\mfp,t_\mfp}^\tau$ over $K$ (Lemma \ref{lem:dio.2})
and the rational function $(X_1+X_2)(1+t_\mfp X_5^{e+1}(X_3+X_4))^{-1}$.
\end{proof}

\begin{proposition}\label{prp:D.property.2}
If $A,B$ are $K$-algebras as in Proposition \ref{prp:AB.1}, then
\begin{align*}
D_{\mfp,t_\mfp,A,B}^{\tau}(E)&\subseteq \mathcal{O}_E
\end{align*}
for every finite extension $E/K_\mfp$
of relative type at most $\tau$.
\end{proposition}
\begin{proof}
By Proposition \ref{prp:AB.2}(i), we have 
$T_{A}(E)+T_{B}(E)=\mathcal{O}_E$.
Since also
$\gamma_{\mfp,t_\mfp}^\tau(E)\subseteq\mathcal{O}_E$
and
$1+t_{\mfp}\mathcal{O}_E\subseteq\mathcal{O}_E^\times$, we have $D_{\mfp,t_\mfp,A,B}^{\tau}(E)\subseteq\mathcal{O}_E$, as required.
\end{proof}

\begin{proposition}\label{prp:D.property.1}
If $A,B$ are $K$-algebras as in Proposition \ref{prp:AB.1},
then
$$
    D_{\mfp,t_\mfp,A,B}^{\tau}(L)=R_{\mathfrak{p}}^{\tau}(L)
$$
for every number field $L$ containing $K$.
\end{proposition}
\begin{proof}
By Proposition \ref{prp:D.property.2}, $D_{\mfp,t_\mfp,A,B}^{\tau}(L_\mfP)\subseteq\mathcal{O}_{L_\mfP}$
for every $\mfP\in\mathcal{S}_\mfp^\tau(L)$, hence
$$
    D_{\mfp,t_\mfp,A,B}^\tau(L)\subseteq\bigcap_{\mfP\in\mathcal{S}_\mfp^\tau(L)}\mathcal{O}_{L_\mfP}\cap L=\bigcap_{\mfP\in\mathcal{S}_\mfp^\tau(L)}\mathcal{O}_\mfP=R_\mfp^\tau(L).
$$

To show the other inclusion,  let $r\in R_{\mathfrak{p}}^{\tau}(L)$.
Since $L/K$ is finite, 
the set $\mathcal{S}_\mfp^*(L)$
of primes of $L$ over $\mathfrak{p}$ is finite.
Write 
$\mfP_{1},...,\mfP_{k}\in\mathcal{S}_{\mathfrak{p}}^{\tau}(L)$
for the primes over $\mfp$ of relative type $\leq\tau$,
and
$\mfQ_1,\dots,\mfQ_l$
for the primes over $\mfp$ not of relative type $\leq\tau$.
For each $i\in\{1,...,l\}$,
by Lemma \ref{lem:gamma} there exists $z_{i}$ such that
\begin{align*}
v_{\mfQ_i}(\gamma_{\mfp, t_\mfp}^{\tau}(z_{i}))&\leq-\frac{1}{e+1}v_{\mfQ_i}(t_{\mathfrak{p}}),
\end{align*}
i.e.
$v_{\mfQ_i}((t_{\mathfrak{p}}\gamma_{\mfp, t_\mfp}^{\tau}(z_i)^{e+1})^{-1})\geq0$.
By weak approximation and continuity of rational functions, 
there exists $z\in L$ such that
$v_{\mfQ_i}((t_{\mathfrak{p}}\gamma_{\mfp, t_\mfp}^{\tau}(z)^{e+1})^{-1})\geq0$
for each $i\in\{1,...,l\}$.
By another application of weak approximation
there exists
$y\in L$
such that\\

\begin{tabular}{rrll}
$v_{\mfQ_i}\big((t_{\mathfrak{p}}\gamma_{\mfp, t_\mfp}^{\tau}(z)^{e+1})^{-1}+y\big)$&$\geq$&$\max\{0,-v_{\mfQ_i}(rt_{\mfp}\gamma_{\mfp, t_\mfp}^{\tau}(z)^{e+1})\}$,&\quad $i=1,\dots,l$,\\
$v_{\mfP_i}(y)$&$\geq$&$0$,&\quad $i=1,\dots,k$.\\
\end{tabular}

\vspace{0.2cm}
\noindent
In particular, $y\in\bigcap_{\mfP\in\mathcal{S}_{\mfp}^{*}(L)}\mathcal{O}_{\mfP}$ 
and $x:=r(1+t_{\mathfrak{p}}\gamma_{\mfp, t_\mfp}^{\tau}(z)^{e+1}y)$
satisfies
$v_{\mfQ_i}(x)\geq0$
for each $i\in\{1,...,l\}$.
As $\mfP_i\in\mathcal{S}_\mfp^\tau(L)$, we have
$r,t_\mfp,\gamma_{\mfp,t_\mfp}^\tau(z),y\in\mathcal{O}_{\mfP_i}$,
hence
$v_{\mfP_i}(x)\geq 0$
for all $i\in\{1,...,k\}$.
Thus 
$x\in\bigcap_{\mfP\in\mathcal{S}_{\mfp}^{*}(L)}\mathcal{O}_{\mfP}$.
As 
$$
 \bigcap_{\mfP\in\mathcal{S}_{\mfp}^{*}(L)}\mathcal{O}_{\mfP}\subseteq T_A(L)+T_B(L)
$$ 
by
Proposition \ref{prp:AB.2}(ii),
we get that 
$$
 r=x(1+t_\mfp\gamma_{\mfp,t_\mfp}^\tau(z)^{e+1}y)^{-1}\in D_{\mfp,t_\mfp,A,B}^{\tau}(L),
$$ 
as required.
\end{proof}

\begin{theorem}\label{thm:main.2}
For every finite place $\mfp$ of a number field $K$
and every $\tau\in\mathbb{N}^2$,
there exists $N\in\mathbb{N}$ such that
$\pi_{\mfp}^{\tau}(L)\leq N$
for every number field $L$ containing $K$.
\end{theorem}

\begin{proof}
We choose algebras $A$ and $B$ over $K$ according to Proposition \ref{prp:AB.1},
and we apply Proposition \ref{prop:compactness_bounded_pi} to the class $\mathcal{K}$ of finite extensions $L/K$ and the diophantine family $D=D_{\mfp,t_\mfp,A,B}^\tau$,
where the two assumptions of Proposition \ref{prop:compactness_bounded_pi}
are verified in
Proposition~\ref{prp:D.property.1} and
Proposition~\ref{prp:D.property.2}, respectively.
\end{proof}

\section{The \texorpdfstring{$(\mfp,\tau)$}{(p,τ)}-Pythagoras number
in finite extensions}
The growth of the (usual) pythagoras number 
is bounded in finite extensions $E/F$ by
\begin{align*}
\pi(E)&\leq [E:F]\cdot\pi(F),
\end{align*}
see \cite[Ch.\ 7 Prop.\ 1.13]{Pfi95}.
We now combine ideas from the proof of Theorem~\ref{thm:main.2}
with techniques for $p$-valuations on general fields
to prove an (inexplicit) analogue of this for the
$(\mfp,\tau)$-Pythagoras number.

As before fix $K$, $\mfp$ and $\tau=(e,f)$ and let $F/K$ be an extension.
We equip $\mcS_{\mfp}^{\tau}(F)$ with the {\em constructible topology},
which by definition has
a basis consisting of the sets 
$$
    \mathcal{S}_\mfp^\tau(F;a) := \{\mfP\in\mathcal{S}_\mfp^\tau(F)\;|\;v_\mfP(a)\geq0\},\quad a\in F
$$
and their complements.
In \cite{ADF1}, we studied approximation theorems for spaces of localities, i.e.~valuations, orderings, and absolute values, on a given field.
We now deduce an approximation theorem in the setting of the space $\mcS_{\mfp}^{\tau}(F)$.

\begin{theorem}\label{thm:approximation}
Let $S_{1},\dots,S_{n}\subseteq\mathcal{S}_{\mfp}^{\tau}(F)$ be disjoint and closed,
let $x_{1},...,x_{n}\in F$,
and let $z_{1},...,z_{n}\in F^{\times}$.
Assume that, for any $\mfP_{i}\in S_{i}$ and $\mfP_{j}\in S_{j}$,
if the valuation $w$ is the finest common coarsening of $v_{\mfP_{i}}$ and $v_{\mfP_{j}}$,
then
$w(x_{i}-x_{j})\geq w(z_{i})=w(z_{j})$.
Then there exists $x\in F$ with
$$
	v_{\mfQ}(x-x_i) > v_{\mfQ}(z_i) \text{ for all }\mfQ \in S_i, \mbox{ for } i=1,\dots,n.
$$
\end{theorem}
\begin{proof}
Corollary 5.5 of
\cite{ADF1} is a similar statement in which $\mcS_{\mfp}^{\tau}(F)$ is replaced by a space ${\rm S}_{\pi}^{e}(F)$,
for $\pi\in F^{\times}$ and $e\in\mathbb{N}$,
By definition (see \cite[Example 2.4]{ADF1}),
${\rm S}_{\pi}^{e}(F)$ is the space
of equivalence classes of valuations $v$ on $F$ with value group $\Gamma_{v}$, which has $\mathbb{Z}$ as a convex subgroup and $0<v(\pi)\leq e$.
We note that $\mathcal{S}_\mfp^\tau(F)\subseteq{\rm S}_{t_\mfp}^e(F)$,
and if we equip ${\rm S}_{t_\mfp}^e(F)$ with its own constructible topology (see \cite[Section 2]{ADF1}) then 
$\mathcal{S}_\mfp^\tau(F)$ is a closed subspace:
By \cite[Lemma 6.2]{PR84}, $\mathcal{S}_\mfp^\tau(F)$
is the intersection over all sets
$\{v\in{\rm S}_{t_\mfp}^{e}(F):v(a)\geq0\}$ for 
$a\in\mathcal{O}_\mfp\cup\gamma_{\mfp,t_\mfp}^\tau(F)$.
Therefore, each $S_{i}$ 
is also a closed subset of $\mathrm{S}_{t_\mfp}^{e}(F)$
and so we may obtain the required element $x\in F$ by an application of \cite[Corollary 5.5]{ADF1}.
\end{proof}

\begin{lemma}\label{lem:omega}
Let $\tau\leq\tau'\in\mathbb{N}^2$.
There is a rational function
$\omega_{\tau,\tau'}\in\mathbb{Q}(t_{\mathfrak{p}})(X)$
such that 
$v_{\mfP}(\omega_{\tau,\tau'}(x))>0$
for all $x\in F$ and $\mfP\in\mathcal{S}_\mfp^{\tau'}(F)$,
and moreover
$v_{\mfP}(\omega_{\tau,\tau'}(x))=1$
if $v_\mfP(x)=1$ and $\mfP$ is of exact relative type $\tau$ over $\mfp$.
\end{lemma}
\begin{proof}
Write $\tau'=(e',f')$.
By Dirichlet's theorem on primes in arithmetic progressions
there exists $k\in\mathbb{N}$ such that $\ell:=1+ke$
is a prime number and $\ell>e'$. Let $\beta(X)=t_\mfp^{-k}X^\ell$.
For every $\mfP\in\mathcal{S}_\mfp^{\tau'}(F)$ and $x\in F$ we have
$v_\mfP(\beta(x))=\ell v_\mfP(x)-kv_\mfP(t_\mfp)$,
which is non-zero (since $\ell>k$ and $\ell>e'\geq v_\mfP(t_\mfp)$
imply $\ell\nmid kv_\mfP(t_\mfp)$),
and equals $1$ if $v_\mfP(x)=1$ and $v_\mfP(t_\mfp)=e$.
Thus $\omega_{\tau,\tau'}(X)=(\beta(X)+\beta(X)^{-1})^{-1}$ satisfies the claim.
\end{proof}

\begin{lemma}\label{lem:rho}
There is a rational function $\rho_{\tau}\in\mathbb{Q}(X)$ such that
for all $\mfP\in\mcS_{\mfp}^{\tau}(F)$
and all $x\in F$
we have
$$
v_{\mfP}(\rho_{\tau}(x))\;
\left\{
\begin{array}{ll}
=0,&\mbox{if }v_{\mfP}(x)=0,\\
>0,&\mbox{if }v_{\mfP}(x)\neq0,
\end{array}\right.
$$
and if $v_{\mfP}(x)=0$
then
$\res_{\mfP}(\rho_{\tau}(x))=\res_{\mfP}(x)$.
\end{lemma}

\begin{proof}
Write $\rho_{\tau}(X)=X(X^{q^f}-X+1)^{-1}$.
Let $\mfP\in\mcS_{\mfp}^{\tau}(F)$
and let $x\in F$.
If $v_{\mfP}(x)<0$ then $v_{\mfP}(x^{q^{f}}-x+1)=q^{f}v_{\mfP}(x)<0$,
and so $v_{\mfP}(\rho_{\tau}(x))=(1-q^{f})v_{\mfP}(x)>0$.
On the other hand, if $v_{\mfP}(x)>0$ then $v_{\mfP}(x^{q^{f}}-x+1)=0$,
so $v_{\mfP}(\rho_{\tau}(x))=v_{\mfP}(x)>0$.
Finally, if $v_{\mfP}(x)=0$ then
\begin{align*}
    \res_{\mfP}(x^{q^{f}}-x+1)&=\res_{\mfP}(x)^{q^{f}}-\res_{\mfP}(x)+1
=1
\neq0,
\end{align*}
and in particular $v_{\mfP}(x^{q^{f}}-x+1)=0$.
Therefore $v_{\mfP}(\rho_{\tau}(x))=0$ and
$\res_{\mfP}(\rho_{\tau}(x))=\res_{\mfP}(x)$.
\end{proof}

\begin{proposition}\label{prp:finding.y}
Let $\tau\leq\tau'=(e',f')$
and
let $S_{0}$ denote an open-closed subset
of $\mcS_{\mfp}^{\tau'}(F)$
such that
$\mcS_{\mfp}^{\tau}(F)\subseteq S_{0}$.
There exists $y\in F$ such that
\begin{align*}
v_{\mfP}(\gamma_{\mfp, t_\mfp}^{\tau}(y))
\left\{
\begin{array}{ll}
\in[0,e'eq^{f}],&\mbox{if }\mfP\in S_{0},\\
<0,&\mbox{if }\mfP\in\mcS_{\mfp}^{\tau'}(F)\setminus S_{0}.
\end{array}\right.
\end{align*}
\end{proposition}

\begin{proof}
For each $\mfP\in\mcS_{\mfp}^{\tau'}(F)\setminus S_{0}$,
we choose $y_{\mfP}\in F$ as follows.
First, if the relative type of $\mfP$ is exactly $\tau''=(e'',f'')$ with $e''>e$, 
then let $t_{\mfP}$ be a uniformizer of $v_{\mfP}$ and set
$y_{\mfP}=\omega_{\tau'',\tau'}(t_{\mfP})$.
By Lemma \ref{lem:omega}, $v_{\mfP}(y_{\mfP})=1$;
and
by Lemma \ref{lem:gamma},
$v_{\mfP}(\gamma_{\mfp, t_\mfp}^{\tau}(y_{\mfP}))<0$.
Also, for all $\mfQ\in\mathcal{S}_{\mfp}^{\tau'}(F)$ we have $v_{\mfQ}(y_{\mfP})>0$.
In particular, $y_{\mfP}\in R^{\tau'}_{\mfp}(F)$.

On the other hand, if the relative type of $\mfP$ is exactly $\tau''=(e'',f'')$ with $f''\nmid f$,
then let $a_{\mfP}$ with $v_{\mfP}(a_{\mfP})=0$ and $\res_\mfP(a_{\mfP})$ a generator of $Fv_{\mfP}$,
and set $y_{\mfP}=\rho_{\tau'}(a_{\mfP})$.
By Lemma \ref{lem:rho}, $v_{\mfP}(y_{\mfP})=0$ and $\res_\mfP(y_{\mfP})$ is a generator of $Fv_{\mfP}$.
By Lemma \ref{lem:gamma},
we have 
$v_{\mfP}(\gamma_{\mfp, t_\mfp}^{\tau}(y_{\mfP}))<0$.
Also, for all $\mfQ\in\mathcal{S}_{\mfp}^{\tau'}(F)$ we have $v_{\mfQ}(y_{\mfP})\geq0$,
i.e.~$y_{\mfP}\in R^{\tau'}_{\mfp}(F)$.

In either case, we have chosen $y_{\mfP}\in R_{\mfp}^{\tau'}(F)$ such that
$v_{\mfP}(\gamma_{\mfp, t_\mfp}^\tau(y_{\mfP}))<0$.
Next we make use of the compactness of $\mcS_{\mfp}^{\tau'}(F)$.
For $y\in F$, we let
$$
	S_{y}=\{\mfP\in\mcS_\mfp^{\tau'}(F)\;|\;v_{\mfP}(\gamma_{\mfp, t_\mfp}^\tau(y))<0\}.
$$
Each $S_{y}$ is an open-closed subset of $\mathcal{S}_{\mfp}^{\tau'}(F)$.
By our choice of the elements $y_{\mfP}$, the family
$$
	\left\{S_{y_{\mfP}}\setminus S_{0}\;:\;\mfP\in\mcS_{\mfp}^{\tau'}(F)\setminus S_{0}\right\}
$$
is an open covering of $\mcS_{\mfp}^{\tau'}(F)\setminus S_{0}$.
So by compactness there exist
$\mfP_{1},\dots,\mfP_{n}\in\mcS_{\mfp}^{\tau'}(F)\setminus S_{0}$
such that with $S_{i}':=S_{y_{\mfP_{i}}}$,
we have
$$
    \mcS_\mfp^{\tau'}(F)=S_0\cup S_{1}'\cup\dots\cup S_{n}'.
$$
Choose open-closed sets
$S_{1}\subseteq S_{1}',\dots,S_{n}\subseteq S_{n}'$
such that
$$
    \mcS_{\mfp}^{\tau'}(F) = S_{0}\sqcup S_{1}\sqcup\dots\sqcup S_{n}
$$
is a partition.
We seek to apply Theorem \ref{thm:approximation}
to the sets $S_0,S_1,\dots,S_n$,
the elements
$x_0=t_\mfp^{-1}$,
$x_1=y_{\mfP_1},\dots,x_n=y_{\mfP_n}$
and $z_0=t_\mfp,\dots,z_n=t_\mfp$.
To verify that the hypothesis of the theorem holds,
we argue as follows:
let $w$ be any valuation on $F$ that is a common coarsening of valuations
$v_{\mfP}$ and $v_{\mfQ}$
corresponding to primes $\mfP\in S_{i}$ and $\mfQ\in S_{j}$,
for $i\neq j$.
Note that $w$ is a proper coarsening of these
valuations since $S_{i}$ and $S_{j}$ are disjoint
and $v_\mfP$, $v_\mfQ$ are incomparable.
Then $w(z_{i})=w(z_{j})=0$ and $w(x_{i}-x_{j})\geq0$.
Therefore, by Theorem \ref{thm:approximation}, there exists
$y\in F$ such that
$$
    v_{\mfP}(y-x_i)>v_{\mfP}(t_\mfp),
$$
for each $\mfP\in S_i$ and each $i$.
In particular, for $\mfP\in S_0$ we have that $v_{\mfP}(y)=-v_{\mfP}(t_\mfp)<0$, hence
$$
    v_{\mfP}(\gamma_{\mfp, t_\mfp}^\tau(y))=eq^{f}v_{\mfP}(t_{\mfp})-v_{\mfP}(t_{\mfp})=(eq^{f}-1)v_{\mfP}(t_\mfp)\in\{0,\dots,e'eq^{f}\},
$$
cf.~Lemma \ref{lem:gamma.1}.
On the other hand,
for $\mfQ\in S_{i}$,
with $i>0$,
we get that $v_{\mfQ}(y-y_{\mfP_i})>v_{\mfQ}(t_\mfp)$.
Since we have
$v_{\mfQ}(\gamma_{\mfp, t_\mfp}^\tau(y_{\mfP_i}))<0$,
then $v_{\mfQ}(\gamma_{\mfp, t_\mfp}^\tau(y))<0$ by
Lemma \ref{lem:gamma.3}.
\end{proof}

Fix $n, m \in \mathbb{N}$
and let
$\tau'=(e',f')$,
where
$e'=me$
and
$f'=m!f$.
Let $\mathcal{E}$ be the class of fields $E$ which contain some $F/K$ with $[E:F]=m$ and $\pi_\mfp^\tau(F)=n$.
We adapt the arguments of Section \ref{section:number.fields}
in order to show that $\pi_{\mfp}^{\tau}(E)$ is bounded by a function of $m,n$.
We let
\begin{align*}
    D_{\mfp,m,n}^{\tau,(1)}(F)&:=\big\{x\in F\;\big|\;\exists a_{0},\ldots,a_{m-1}\in R_{\mfp,n}^{\tau}(F)\;:\;x^{m}+a_{m-1}x^{m-1}+\ldots+a_{0}=0\big\},
\end{align*}
and
\begin{align*}
    D_{\mfp,m,n}^{\tau,(2)}(F)&:=\bigg\{\frac{a}{1+t_{\mfp}\gamma_{\mfp, t_\mfp}^{\tau}(y)^{e'}b}\;\bigg|\;a,b\in D_{\mfp,m,n}^{\tau,(1)}(F),y\in F,
    \gamma_{\mfp, t_\mfp}^\tau(y) \neq \infty, 1+t_\mfp \gamma_{\mfp, t_\mfp}^\tau(y)^{e'} b \neq 0 \bigg\}.
\end{align*}

\begin{lemma}\label{lem:D.2.dio}
Both
$D_{\mfp,m,n}^{\tau,(1)}$ and $D_{\mfp,m,n}^{\tau,(2)}$
are $1$-dimensional diophantine families over $K$.
\end{lemma}
\begin{proof}
This is very similar to
Lemma \ref{lem:D.dio}.
This time we use the fact that
$R_{\mfp,n}^{\tau}$ is a $1$-dimensional diophantine family over $K$,
as seen in Example \ref{ex:Gamma.n.diophantine}.
From this is immediately follows that $D_{\mfp,m,n}^{\tau,(1)}$ is a $1$-dimensional diophantine family over $K$.
To see that $D_{\mfp,m,n}^{\tau,(2)}$ is a $1$-dimensional diophantine family over $K$ we now apply Lemma \ref{lem:dio.rational}
to the $3$-dimensional diophantine family $D_{\mfp,m,n}^{\tau,(1)}\times D_{\mfp,m,n}^{\tau,(1)}\times\gamma_{\mfp,t_\mfp}^\tau$ and the rational function
$X_1(1+t_\mfp X_3^{e'}X_2)^{-1}$.
\end{proof}

\begin{proposition}\label{prp:D.2.property.2}
For every $E\supseteq K$ we have
$D_{\mfp,m,n}^{\tau,(2)}(E)\subseteq R_{\mfp}^{\tau}(E)$.
\end{proposition}
\begin{proof}
Since $R_{\mfp}^{\tau}(E)$ is integrally closed in $E$
and $R_{\mfp,n}^\tau(E)\subseteq R_\mfp^\tau(E)$,
we have
$D_{\mfp,m,n}^{\tau,(1)}(E)\subseteq R_{\mfp}^{\tau}(E)$.
Let
$\mfP\in\mcS_{\mfp}^{\tau}(E)$.
Then $v_{\mfP}(t_{\mfp})>0$.
Furthermore,
for $y\in E$
and $b\in R_{\mfp}^{\tau}(E)$,
we have
$v_{\mathfrak{P}}(\gamma_{\mfp, t_\mfp}^{\tau}(y)^{e'}b)\geq0$,
hence
$v_{\mfP}(1+t_{\mfp}\gamma_{\mfp, t_\mfp}^{\tau}(y)^{e'}b)=0$.
Therefore
elements of the form
$a(1+t_{\mfp}\gamma_{\mfp, t_\mfp}^{\tau}(y)^{e'}b)^{-1}$
are contained in $R_{\mathfrak{p}}^{\tau}(E)$,
where $a,b\in D_{\mfp,m,n}^{\tau,(1)}(E)$ and $y\in E$.
This establishes $D_{\mfp,m,n}^{\tau,(2)}(E)\subseteq R_{\mfp}^{\tau}(E)$.
\end{proof}

\begin{lemma}\label{lem:D.2.i}
For every $E\in\mathcal{E}$ we have
$R_{\mfp}^{\tau'}(E)\subseteq D_{\mfp,m,n}^{\tau,(1)}(E)$.
\end{lemma}
\begin{proof}
Choose $F$ such that 
$[E:F]=m$ and $\pi_\mfp^\tau(F)=n$,
although the choice of $F$ will not matter.
Let $S$ be the set of primes of $E$
(of arbitrary type) lying over elements of
$\mathcal{S}_{\mathfrak{p}}^{\tau}(F)$.
By our choice of $\tau'$,
we have
$S\subseteq\mcS_{\mfp}^{\tau'}(E)$.
If we denote by $A$ the integral closure of
$R_{\mathfrak{p}}^{\tau}(F)$ in $E$,
then
$A$ is the holomorphy ring corresponding to $S$ and we have
$$
    R_{\mathfrak{p}}^{\tau'}(E)\subseteq A\subseteq R_{\mathfrak{p}}^{\tau}(E).
$$
Since $\pi_{\mfp}^{\tau}(F)=n$, we have
$R_{\mfp}^{\tau}(F)=R_{\mfp,n}^{\tau}(F)$;
and trivially
$R_{\mfp,n}^{\tau}(F)\subseteq R_{\mfp,n}^{\tau}(E)$.
As the degree of the extension $E/F$ is $m$,
$D_{\mfp,m,n}^{\tau,(1)}(E)$
contains the integral closure of $R_{\mfp}^{\tau}(F)$ in $E$, which is $A$.
In particular $R_{\mfp}^{\tau'}(E)\subseteq D_{\mfp,m,n}^{\tau,(1)}(E)$.
\end{proof}

\begin{proposition}\label{prp:D.2.property.1}
For every $E\in\mathcal{E}$ we have
$D_{\mfp,m,n}^{\tau,(2)}(E)=R_{\mfp}^{\tau}(E)$.
\end{proposition}
\begin{proof}
In view of 
Proposition \ref{prp:D.2.property.2},
it only remains to show that
$R_{\mfp}^{\tau}(E)\subseteq D_{\mfp,m,n}^{\tau,(2)}(E)$.
Let $x\in R_{\mfp}^{\tau}(E)$.
In fact, we aim to find
$b\in R_{\mfp}^{\tau'}(E)$ and $y\in E$ with
$$
    x(1+t_{\mfp}\gamma_{\mfp, t_\mfp}^{\tau}(y)^{e'}b)\in R_{\mfp}^{\tau'}(E),
$$
which we will do by applying Theorem \ref{thm:approximation}.
As $R_\mfp^{\tau'}(E)\subseteq D_{\mfp,m,n}^{\tau,(1)}$ by 
Lemma \ref{lem:D.2.i},
this will show that $x\in D_{\mfp,m,n}^{\tau,(2)}(E)$.
We define the sets
\begin{flalign*}
&&S_{0}&:=\{\mathfrak{P}\in\mathcal{S}_{\mathfrak{p}}^{\tau'}(E)\;|\;v_{\mathfrak{P}}(x)\geq0\}&\\
\text{and}&& S_{1}&:=\mathcal{S}_{\mathfrak{p}}^{\tau'}(E)\setminus S_{0}.&
\end{flalign*}
Note that $S_{0}$ and $S_{1}$ are
open-closed in $\mcS_{\mfp}^{\tau'}(E)$
and
$S_{1}\cap\mathcal{S}_{\mathfrak{p}}^{\tau}(E)=\emptyset$.
We find a suitable element $y\in E$
by a direct application of
Proposition \ref{prp:finding.y}:
we obtain $y\in E$ such that
\begin{align*}
v_{\mfP}(\gamma_{\mfp, t_\mfp}^{\tau}(y))
\left\{
\begin{array}{ll}
\in[0,e'eq^{f}],&\mbox{if }\mfP\in S_{0},\\
<0,&\mbox{if }\mfP\in S_{1}.
\end{array}\right.
\end{align*}
We obtain a suitable $b\in E$ by solving a more straightforward approximation problem: By Theorem \ref{thm:approximation},
there exists $b\in R_{\mfp}^{\tau'}(E)$ such that
\begin{align*}
    v_{\mfP}(b)&\geq0,&
    \mbox{if }\mfP\in S_{0},\\
    \text{and}\hspace{8mm}
    v_{\mfP}(b+t_{\mfp}^{-1}\gamma_{\mfp, t_\mfp}^{\tau}(y)^{-e'})&\geq v_{\mfP}(x^{-1}t_{\mfp}^{-1}\gamma_{\mfp, t_\mfp}^{\tau}(y)^{-e'}),
    &\mbox{if }\mfP\in S_{1}.
\end{align*}
Indeed,
if a valuation $w$ on $E$ coarsens $v_\mfP$ and $v_\mfQ$ for $\mfP\in S_{0}$ and $\mfQ\in S_{1}$, 
$v_\mfP(x)\geq0$ and $v_\mfQ(x)<0$ imply that $w(x)=0$,
and $v_\mfP(\gamma_{\mfp,t_\mfp}^\tau(y))\in[0,e'eq^f]$
implies that $w(\gamma_{\mfp, t_\mfp}^{\tau}(y))=0$.
Therefore also
$w(t_{\mfp}\gamma_{\mfp, t_\mfp}^{\tau}(y)^{e'})=0$
and
$w(xt_{\mfp}\gamma_{\mfp, t_\mfp}^{\tau}(y)^{e'})=0$.
In particular, the hypothesis of the theorem
is satisfied,
and the $b\in E$ so obtained 
lies in $R_\mfp^{\tau'}(E)$.

For $\mfP\in S_{0}$, we have $v_{\mfP}(t_{\mfp}^{-1}\gamma_{\mfp, t_\mfp}^{\tau}(y)^{-e'})<0$,
hence
\begin{align*}
v_{\mfP}(b+t_{\mfp}^{-1}\gamma_{\mfp, t_\mfp}^{\tau}(y)^{-e'})
\left\{
\begin{array}{ll}
=v_{\mfP}(t_{\mfp}^{-1}\gamma_{\mfp, t_\mfp}^{\tau}(y)^{-e'}),&\mbox{if }\mfP\in S_{0},\\
\geq v_{\mfP}(x^{-1}t_{\mfp}^{-1}\gamma_{\mfp, t_\mfp}^{\tau}(y)^{-e'}),&\mbox{if }\mfP\in S_{1},
\end{array}\right.
\end{align*}
i.e.~
\begin{align*}
\begin{array}{rlr}
v_{\mathfrak{P}}(1+t_{\mfp}\gamma_{\mfp, t_\mfp}^{\tau}(y)^{e'}b)&=0,
&\mbox{if }
\mathfrak{P}\in S_{0},\\
v_{\mathfrak{P}}(x(1+t_{\mfp}\gamma_{\mfp, t_\mfp}^{\tau}(y)^{e'}b))
&\geq0,
&\mbox{if }
\mathfrak{P}\in S_{1}.
\end{array}
\end{align*}
Since $v_\mfP(x)\geq0$ for $\mfP\in S_0$, we obtain that
$x(1+t_{\mfp}\gamma_{\mfp, t_\mfp}^{\tau}(y)^{e'}b)\in R_{\mathfrak{p}}^{\tau'}(E)$.
\end{proof}

\begin{theorem}\label{thm:p_pythagoras_in_fin_extn}
There is a function
$\alpha_{\mathfrak{p}}^{\tau}:\mathbb{N}\times\mathbb{N}\longrightarrow\mathbb{N}$
such that 
\begin{align*}
\pi_{\mathfrak{p}}^{\tau}(E)&\leq\alpha_{\mathfrak{p}}^{\tau}(\pi_{\mathfrak{p}}^{\tau}(F),[E:F]),
\end{align*}
for every field extension $E/F$ with $\pi_\mfp^\tau(F)<\infty$.
\end{theorem}
\begin{proof}
Let $m,n\in\mathbb{N}$.
We apply Proposition \ref{prop:compactness_bounded_pi}
to the class $\mathcal{E}$
and the diophantine family
$D_{\mfp,m,n,}^{\tau,(2)}$,
where the two assumptions of Proposition \ref{prop:compactness_bounded_pi}
are verified in Proposition
\ref{prp:D.2.property.1}
and Proposition
\ref{prp:D.2.property.2},
respectively.
Thus there exists $N$ such that $\pi_{\mfp}^{\tau}(E)\leq N$
for every $E\in\mathcal{E}$,
so we can choose $\alpha_{\mfp}^{\tau}(n,m)=N$.
\end{proof}

\section{Diophantine holomorphy rings of $p$-valuations}

By definition, in any field $F$ with finite $(\mfp, \tau)$-Pythagoras number the holomorphy ring $R_\mfp^\tau(F)$ is a diophantine subset. 
In this section we generalize this observation, by showing in Corollary~\ref{cor:holomorphy.diophantine} that the same applies to the holomorphy rings associated to arbitrary open-closed subsets of $\mathcal{S}_\mfp^\tau(F)$. Theorem \ref{prp:2.dimensional.diophantine} is a uniform version of this fact.

As a technical tool, it turns out to be useful to extend some of the ideas from diophantine families over fields to commutative algebras which are finite-dimensional vector spaces over fields.
To this end, we introduce a small piece of notation.
Write $X=(X_{1},...,X_{n})$ and $Y=(Y_{1},...,Y_{m})$.
For $f_{1},..,f_{r}\in K[X,Y]$ and for any commutative (unital, associative) $F$-algebra $B$, we write
\begin{align*}
P_{f_{1},...,f_{r}}(B):=\left\{x\in B^{n}\;|\;\exists y\in B^{m}:f_{1}(x,y)=\dots=f_{r}(x,y)=0\right\}.
\end{align*}
The following lemma is straightforward, but we include it for lack of a suitable reference.

\begin{lemma}\label{lem:algebra.quotient.fields}
Let $f_{1},...,f_{r}\in K[X,Y]$ and let $l\in\mathbb{N}$.
Then
\begin{align*}
F^{n}\cap P_{f^{l}_{1},...,f^{l}_{r}}(B)&=\bigcap_{\mathfrak{m}\in\mathrm{MaxSpec}(B)} (F^n \cap P_{f_{1},...,f_{r}}(B/\mathfrak{m})),
\end{align*}
for all extensions $F/K$, and all commutative $F$-algebras $B$ of dimension at most $l$.
Here $F$ is identified with its image in $B$ and $B/\mathfrak{m}$.
\end{lemma}
\begin{proof}
Let $B$ be a commutative $F$-algebra which has dimension at most $l$ as an $F$-vector space. 
As $B$ is finite dimensional, it is Artinian,
hence the Jacobson radical $\mathfrak{j}$ of $B$ is nilpotent
(\cite[Prop.~8.4]{AM}), and therefore more precisely
$\mathfrak{j}^{l}=0$.
Then for all $s\in\{1,\dots,r\}$, all extensions $F/K$, all $a\in F$,  $x\in F^{n}$, and $y\in B^{m}$,
we have
\begin{align*}
f_{s}(x,y)^{l}=0&\Longleftrightarrow f_{s}(x,y+\mathfrak{j})=0\\
&\Longleftrightarrow f_{s}(x,y+\mathfrak{m})=0,\;\forall\mathfrak{m}\in\mathrm{MaxSpec}(B).
\end{align*}
The result now follows from the Chinese Remainder Theorem.
\end{proof}

\begin{lemma}\label{lem:interpret.algebra}
Let $f_{1},...,f_{r}\in K[X,Y]$ and let $k\in\mathbb{N}$.
There exists an $(n+k)$-dimensional diophantine family $D$ over $K$
such that
\begin{align*}
D(F)&=\Big\{(x,z)\in F^n\times F^k\;\Big|\;x\in P_{f_{1},...,f_{r}}(B_{z})\Big\},
\end{align*}
for all extensions $F/K$, and where $B_{z}$ denotes the commutative $F$-algebra
$$F[T]\Big/\Big(T^{k}+\sum_{i=0}^{k-1}z_{i}T^{i}\Big).$$
\end{lemma}

\begin{proof}
In a more advanced way, this construction can be described through the
Weil restriction of the affine variety cut out by the polynomials $f_{1},\dots,f_{r}$,
along the family of schemes
described by the $B_{z}$,
fibred over the parameter space $\mathbb{A}^{k}$.
Alternatively, from a model-theoretic standpoint, one can prove the statement by a quantifier-free interpretation of $B_{z}$ in $F$, uniformly in the parameter tuple $z$.
We give an elementary description instead.

We introduce two new tuples of variables $Z=(Z_i)_{0\leq i<k}$ and
$U=(U_{i,j})_{0\leq i<k,1\leq j\leq m}$.
We write $g(Z,T):=T^{k}+\sum_{i=0}^{k-1}Z_iT^{i}\in K[Z,T]$ and,
for each $s\in\{1,\dots,r\}$,
we let
\begin{align*}
\hat{f}_{s}(X,U,T)&:=f_{s}\Big(X,\sum_{i=0}^{k-1}U_{i,1}T^{i},...,\sum_{i=0}^{k-1}U_{i,m}T^{i}\Big).
\end{align*}
Choose $d\in\mathbb{N}$ to be the maximum of the degrees of the polynomials $\hat{f}_{s}$ in the variable $T$,
and introduce a new tuple of variables
$W=(W_{l})_{0\leq l \leq d}$.
Then, for each $s$, we consider the polynomial
\begin{align*}
\tilde{f}_{s}(X,Z,U,W,T)&:=\hat{f}_{s}(X,U,T)-g(Z,T)\sum_{l=0}^dW_{l}T^{l}.
\end{align*}
Note that $\tilde{f}_{s}(x,z,u,w,T)=0$ for some $w$ if and only if $g(z,T)$ divides $\hat{f}_{s}(x,u,T)$ in $F[T]$.
By taking coefficients with respect to the variable $T$,
we obtain a family of polynomials $h_{s,l}\in K[X,Z,U,W]$,
for $1\leq s\leq r$ and $0\leq l \leq d+k$,
such that
\begin{align*}
\tilde{f}_{s}(X,Z,U,W,T)&=\sum_{l=0}^{d+k}h_{s,l}(X,Z,U,W)T^{l}.
\end{align*}
We may define 
the required $(n+k)$-dimensional diophantine family $D$ over $K$ by writing
\begin{align*}
    D(F)&=\big\{(x,z)\in F^{n} \times F^{k}\;\big|\;\exists u\in F^{km},w\in F^{d+1}:h_{s,l}(x,z,u,w)=0\mbox{ for all }s,l\big\},
\end{align*}
for $F/K$.
\end{proof}

\begin{lemma}\label{lem:kill}
For every field extension $F/K$
and every $a\in F$, we have
\[
 \mathcal{S}_\mfp^\tau(F;a) = \bigcup_{\mathfrak{m}\in{\rm MaxSpec}(B_{a})}{\rm res}_{(B_{a}/\mathfrak{m})/F}(\mathcal{S}_\mfp^\tau(B_{a}/\mathfrak{m})),
\]
where 
${\rm res}_{E/F}$ denotes restriction of primes from $E$ to $F$, and
$B_a$ is the commutative $F$-algebra
$$
 F[T]\Big/\Big(t_\mfp a^e ((T^{q^f}-T)^2-1) - (T^{q^f}-T)\Big).
$$ 
\end{lemma}

\begin{proof}
Denote
 ${\rm MaxSpec}(B_{a})=\{\mathfrak{m}_1,\dots,\mathfrak{m}_r\}$ and
$E_i=B_{a}/\mathfrak{m}_i$.
Let 
$g_a = t_\mfp a^e ((T^{q^f}-T)^2-1) - (T^{q^f}-T) \in F[T]$
and note that $g_a$ is closely related to $\gamma_{\mfp,t_\mfp}^\tau$.

First let $\mfP\in\mathcal{S}_{\mfp}^{\tau}(E_{i})$ for some $i$.
If $\theta$ denotes the residue of $T$ in $E_i$, we have 
$\gamma_{\mfp,t_\mfp}^\tau(\theta)\in\mathcal{O}_\mfP$
 and therefore 
$v_\mfP({\theta^{q^f}-\theta})>v_\mfP({(\theta^{q^f}-\theta)^2-1})$, 
so since $g_a(\theta) = 0$ we necessarily have 
$v_\mfP(t_\mfp a^e) > 0$ and therefore $v_\mfP(a) \geq0$.

Conversely, let $\mfP \in \mathcal{S}_\mfp^\tau(F; a)$.
Then $g_a\in\mathcal{O}_\mfP[T]$ has a simple zero $T=0$ modulo the maximal ideal of $\mathcal{O}_\mfP$,
which implies that there exists some $i$ and $\mfQ\in\mathcal{S}_\mfp^\tau(E_i)$
with $\mfP={\rm res}_{E_i/F}(\mfQ)$: Indeed,
if $(F',v')$ is a henselization of $(F,v_\mfP)$,
then $v'=v_{\mfP'}$ for a prime $\mfP'$ of $F'$,
and Hensel's lemma in the form \cite[Theorem 4.1.3(4)]{EP05}
shows that $g_a$ has a zero in $F'$,
which induces an $F$-embedding $E_i\rightarrow F'$,
and one can take $\mfQ={\rm res}_{F'/E_i}(\mfP')$.
\end{proof}

\begin{theorem}\label{prp:2.dimensional.diophantine}
For every $N\in\mathbb{N}$ there exists a $2$-dimensional diophantine family $D_{\mfp,N}^{\tau}$ over $K$ such that
$$
    D_{\mfp,N}^{\tau}(F) = \left\{(x,a)\in F^2 \;|\; v_\mfP(x)\geq0\mbox{ for every }\mfP\in\mathcal{S}_\mfp^\tau(F;a) \right\}.
$$
for every extension $F/K$ with $\pi_\mfp^\tau(F)\leq N$.
\end{theorem}

\begin{proof}
Let $l = 2q^f$.
By Theorem \ref{thm:p_pythagoras_in_fin_extn} there exists $N'$ such that
for all $E/F/K$ with $[E:F]\leq l$ and $\pi_\mfp^\tau(F)\leq N$, we have
$\pi_\mfp^\tau(E)\leq N'$,
and so
\begin{align}\label{eq:1}
R_{\mfp}^{\tau}(E)&=R_{\mfp,N'}^{\tau}(E).
\end{align}
By Example \ref{ex:Gamma.n.diophantine}, $R_{\mfp,N'}^{\tau}$ is a $1$-dimensional diophantine family over $K$,
and so we may choose polynomials
$f_{1},...,f_{r}\in K[X,Y_1,\dots,Y_m]$ such that
\begin{align}\label{eq:2}
R_{\mfp,N'}^{\tau}(F)&=\{x\in F\;|\;\exists y\in F^{m}:f_{1}(x,y)=...=f_{r}(x,y)=0\}
\end{align}
for all $F/K$.
For each $F/K$ with $\pi_{\mfp}^{\tau}(F)\leq N$, and each $a\in F$, we have
\begin{equation}
    \begin{aligned}
        &&F\cap P_{f_{1}^{l},...,f_{r}^{l}}(B_{a})&=\bigcap_{\mathfrak{m}\in\mathrm{MaxSpec}(B_{a})} (F\cap P_{f_{1},...,f_{r}}(B_{a}/\mathfrak{m}))&&\text{by Lemma \ref{lem:algebra.quotient.fields},}\\
        &&&=\bigcap_{\mathfrak{m}\in\mathrm{MaxSpec}(B_{a})} (F \cap R_{\mfp}^{\tau}(B_{a}/\mathfrak{m}))&&\text{by \eqref{eq:1} and \eqref{eq:2},}\\
        &&&=\bigcap_{\mfP\in\mathcal{S}_{\mfp}^{\tau}(F;a)}\mathcal{O}_{\mfP}&&\text{by Lemma \ref{lem:kill}},
    \end{aligned}
\label{eq:3}
\end{equation}
where $B_{a}$ is the $l$-dimensional algebra from Lemma \ref{lem:kill}.

By Lemma \ref{lem:interpret.algebra}, we may define a $2$-dimensional diophantine family $D$ over $K$ satisfying
\[ D(F) = \{ (x,a) \in F^2 \;|\ x \in P_{f^l_1, \dotsc, f^l_r}(B_a) \} \]
for every extension $F/K$.
By \eqref{eq:3}, for every $F/K$ with $\pi_\mfp^\tau(F) \leq N$ we in fact have
$$ D(F) = \big\{ (x,a) \in F^2 \;\big|\ x \in \bigcap_{\mfP\in\mathcal{S}_{\mfp}^{\tau}(F;a)}\mathcal{O}_{\mfP} \big\}, $$
proving the claim.
\end{proof}

\begin{corollary}\label{cor:holomorphy.diophantine}
 If $\pi_{\mfp}^{\tau}(F) < \infty$, then for every open-closed set $S \subseteq \mathcal{S}_\mfp^\tau(F)$, the holomorphy ring $\bigcap_{\mfP \in S} \mathcal{O}_\mfP$ is diophantine in $F$.
\end{corollary}
\begin{proof}
  As $S$ is open-closed, it is of the form $\mathcal{S}_\mfp^\tau(F; a)$ for some $a \in F$, see \cite[Lemmas 10.4, 10.5]{Feh13}.
  Hence the claim follows from Theorem \ref{prp:2.dimensional.diophantine} and Lemma \ref{lem:dioph_section}.
\end{proof}

By Example \ref{Ex:PpC} this applies in particular to pseudo p-adically closed fields like $\mathbb{Q}^{{\rm t}p}$, although for such fields there are in fact simpler ways of establishing Theorem \ref{thm:p_pythagoras_in_fin_extn}.

\section*{Acknowledgements}

Some of this work was completed while the authors were participating in the
{\em 
Model Theory, Combinatorics and Valued fields}
trimester at the Institut Henri Poincar\'{e},
and they would like to extend their thanks to the organisers.
They would like to thank Florian Pop for discussions on the $p$-Pythagoras number.
The results of Section 4 
were in this generality first obtained, in a different formulation, in P.~D.'s doctoral 
thesis \cite{thesis}, during the research for which he was supported by Merton 
College Oxford and the University of Oxford Clarendon Fund.
S.~A.\ was also supported by The Leverhulme Trust
under grant RPG-2017-179.
A.~F.~was funded by the Deutsche Forschungsgemeinschaft (DFG) - 404427454.

\bibliographystyle{plain}

\end{document}